\pdfoutput=1
\documentclass[11pt,letterpaper]{article}

\usepackage[english]{babel}
\usepackage[T1]{fontenc}
\usepackage{stmaryrd}
\usepackage{algorithmic}
\usepackage{amssymb}
\usepackage{marginnote}
\usepackage[ruled,vlined,linesnumbered]{algorithm2e}
\usepackage{booktabs,tabularx,threeparttable}
\usepackage[normalem]{ulem}
\usepackage{tikz}
\usepackage{enumerate}
\usetikzlibrary{fit,backgrounds,positioning,shapes.geometric}
\usepackage{fullpage}
\usepackage{mathrsfs}
\usepackage{amsmath}
\usepackage{bbm}
\usepackage{graphicx}
\usepackage{stackengine}
\usepackage{amsthm}
\usepackage{subcaption}
\usepackage{authblk}
\usepackage{thmtools}
\usepackage{thm-restate}
\usepackage[disable]{todonotes}
\usepackage{fancyhdr,lastpage}
\usepackage{amsfonts}
\usepackage[dvipsnames]{xcolor}
\usepackage[colorlinks=true,allcolors=blue]{hyperref}
\usetikzlibrary{calc, intersections, positioning}

\makeatletter
\renewcommand{\maketitle}{%
\begin{center}
{\large\bfseries \@title\par}
\vspace{0.5em}
{\normalsize \@author\par}
\end{center}
}
\renewenvironment{abstract}{%
\par\normalfont\normalsize
\begin{center}\begin{minipage}{0.9\linewidth}\small\noindent\textup{Abstract. }\ignorespaces}{\end{minipage}\end{center}\par}

\def\@seccntformat#1{\csname the#1\endcsname.\ }

\renewcommand\section{\@startsection{section}{1}{\z@}%
{2.0ex plus .5ex minus .2ex}{1.0ex plus .2ex}%
{\normalfont\large\centering}}

\renewcommand\subsection{\@startsection{subsection}{2}{\z@}%
{1.75ex plus .5ex minus .2ex}{-1em}%
{\normalfont\normalsize\bfseries}}
\makeatother

\newtheorem{thm}{Theorem}[section]

\newtheorem{lm}[thm]{Lemma}

\newtheorem{prop}[thm]{Proposition}

\theoremstyle{remark}

\theoremstyle{plain}

\newcommand{\set}[1]{\left\{#1\right\}}

\newcommand{\abs}[1]{\left| #1 \right|}
\newcommand{\norm}[1]{\left|\left| #1 \right|\right|}

\newcommand{\R}{\mathbb{R}}

\newcommand{\N}{\mathbb{N}}
\newcommand{\C}{\mathbb{C}}
\newcommand{\F}{\mathbb{F}}
\newcommand{\B}{\mathbb{B}}
\newcommand{\Z}{\mathbb{Z}}

\begin{document}

\title{SCHR\"{O}DINGER OPERATORS WITH GENERIC POTENTIALS ACHIEVE MAXIMAL RESONANCE DENSITY}
\author{TRAVIS CUNNINGHAM}
\date{}
\maketitle

\begin{abstract}
We show that for a generic real or complex-valued compactly supported potential, the corresponding Schr\"{o}dinger operator achieves maximal resonance density, in the sense that its integrated resonance counting function achieves the optimal asymptotic upper bound. For odd dimensions this follows from results of Dinh-Vu once we adapt an argument of Christiansen-Hislop. The proof for even dimensions constitutes the bulk of the paper, and we prove several new results on resonances which have analogues in the odd dimensional case. This includes a sharp upper bound on the integrated resonance counting function for any compactly support potential, a proof that the characteristic function of a ball has resonance counting function which achieves the optimal upper bound, and an even-dimensional analogue of the result of Dinh-Vu on asymptotics of the resonance counting functions for complements of pluripolar subsets of analytic families of potentials. We use the characterization of resonances as zeros of certain Fredholm determinant functions related to the scattering matrix, allowing us to apply techniques and results from the theories of one and several complex variables. Our proof that the characteristic function of a ball has counting function achieving the optimal upper bound uses the uniform asymptotics of Bessel functions and follows ideas of Zworski, Christiansen-Hislop, and Dinh-Vu. 
\end{abstract}

\section{INTRODUCTION}

For $V \in L_c^\infty$, the resolvent $R_V(\lambda) := (-\Delta + V - \lambda^2)^{-1}$ has meromorphic continuation as an operator taking $L_c^2(\R^d)$ to $H_{\operatorname{loc}}^2(\R^d)$; this meromorphic continuation is to $\C$ when $d$ is odd, and to the logarithmic cover of $\C$, 
\begin{equation*}
    \Lambda := \exp^{-1}(\C \backslash \set{0}),
\end{equation*}
when $d$ is even. The poles of this meromorphic continuation are called resonances, which can be thought of as generalizing eigenvalues to physical systems that allow energy to escape to infinity. We refer the interested reader to the monograph \cite{8} for a thorough introduction to the theory of resonances. In this paper, we study the asymptotic distribution of resonances for a generic choice of the potential $V$.

Let $d \geq 1$ be odd and let $n_V(r)$ denote the number of resonances with modulus $\leq r$ counted with multiplicity. For $d = 1$, Zworski shows in \cite{23} that
\begin{equation*}
    n_V(r) = \frac{4}{\pi} ar + o(r) \quad \text{as } r \to \infty \tag{1.1} \label{eq:101}
\end{equation*}
where $2a = \operatorname{diam}(\operatorname{supp} V)$ (see also \cite{9}, \cite{17}, \cite{R}, \cite{CC}). For odd $d \geq 3$, no analogue of \eqref{eq:101} is known, and in fact Christiansen shows in \cite{2} that there are complex-valued potentials with no resonances at all. However, if we set 
\begin{equation*}
    N_V(r) := \int_0^r (1/t)(n_V(t) - n_V(0))dt, 
\end{equation*}
then Dinh-Vu, \cite[Theorem 3.1]{7}, improves the earlier results of Zworski, \cite{25}, and Stefanov, \cite{19}, to show that 
\begin{equation*}
    N_V(r) \leq \frac{c_d a^d}{d}r^d + O(r^{d-1} \log r) \quad \text{as } r\to \infty \tag{1.2} \label{eq:102}
\end{equation*}
where $2a = \operatorname{diam}(\operatorname{supp}V)$ and $c_d$ is the explicit dimensional constant defined in \eqref{eq:204}. Moreover, \cite{24}, \cite{19} show that this bound is sharp, as there exist radial potentials with matching lower bound. As is the case with many settings in the theory of resonances, general lower bounds remain an open problem. This is the primary motivation for this paper and our first result shows that in fact the optimal upper bound in \eqref{eq:102} is achieved for "most of" the potentials supported in the ball of radius $a$. 

To state the result precisely, recall that if $X$ is a complete metric space, we call a dense $G_\delta$ set $S \subset X$ \emph{Baire typical} (see \cite{16}); we call a property \emph{generic} if it is satisfied for all elements in a Baire typical subset. Moreover, we define $\mathbb B_a := \set{x \in \R^d : \abs{x} < a}$ and for $\mathbb F = \R$ or $\C$, let $L^\infty(\mathbb B_a; \mathbb F) = \set{f \in L^\infty(\R^d; \mathbb F) : \operatorname{supp}f \subset \overline{\mathbb B}_a}$.

\begin{thm}\label{thm:1.1}
    Let $d \geq 3$ be odd, $a > 0$, and $\F = \R$ or $\C$. Then the set 
    \begin{equation*}
        \mathfrak M_{\F, a} := \set{V \in L^\infty(\mathbb B_a; \F) : \varlimsup_{r\to \infty} \frac{N_V(r)}{r^d} = \frac{c_da^d}{d}}
    \end{equation*}
    is Baire typical in $L^\infty(\B_a; \F)$.
\end{thm}

The above theorem continues a string of related results, see \cite{5}, \cite{3}, \cite{7}, and \cite{Dn}. The proof follows almost entirely from results of Dinh-Vu, \cite{7}; our contribution is only to adapt the argument of Christiansen-Hislop, \cite{5}, and to show how combining that argument with the results of Dinh-Vu, \cite{7}, implies Theorem \ref{thm:1.1}.

The bulk of this paper concerns the even-dimensional analogue of Theorem \ref{thm:1.1}. When the dimension is even and resonances lie in $\Lambda$, we count resonances on each open sheet $\Lambda_m, m \in \Z$ consisting of $\lambda \in \Lambda$ with $m \pi < \arg \lambda < (m+1) \pi$. Let 
\begin{equation*}
    n_{V, m}(r) := \set{\lambda \text{ a resonance } : \lambda \in \Lambda_m, 0 < \abs{\lambda} < r}
\end{equation*}
and 
\begin{equation*}
    N_{V, m}(r) := \int_0^r (1/t)n_{V, m}(t)dt.
\end{equation*}
Compared with the counting functions in odd dimensions, here we avoid the zero resonance which is more complicated to describe in even dimensions (see \cite[Section 2]{4} and references therein). Luckily, we may ignore it since we are only concerned with the asymptotic growth of $n_{V, m}(r)$, $N_{V, m}(r)$. We have the following analogue of Theorem \ref{thm:1.1}:

\begin{thm}\label{thm:1.2}
    Let $d \geq 2$ be even, $a > 0$, and $\F = \R$ or $\C$. Then the set 
    \begin{equation*}
        \mathfrak M_{\F, a} := \set{V \in L^\infty(\B_a, \F) : \text{For each } m \in \Z \backslash \set{0}, \varlimsup_{r \to \infty} \frac{N_{V, m}(r)}{r^d} = \frac{c_da^d}{d}}
    \end{equation*}
    is Baire typical in $L^\infty(\B_a, \F)$. 
\end{thm}

Together with the optimal upper bound in Theorem \ref{thm:1.3} below, this theorem and Theorem \ref{thm:1.1} show that in any dimension, Schr\"{o}dinger operators with generic potentials achieve the maximal possible density of resonances.

Sharpening an earlier result of Intissar, \cite{10}, Vodev gave an upper bound in even dimensions for the order of the counting functions,
\begin{equation*}
    \varlimsup_{r \to \infty} \frac{\log n_{V, m}(r)}{\log r} \leq d, \tag{1.3} \label{eq:103} 
\end{equation*}
see \cite{21}, \cite{22}. Christiansen-Hislop, \cite{6}, show that certain radial potentials have counting functions achieving this optimal order. Using this they prove a result related to our Theorem \ref{thm:1.2}, showing that for $d \geq 2$ even, a fixed compact set $K \subset \R^d$ with nonempty interior, and $\F = \R$ or $\C$, the set 
\begin{equation*}
    \set{V \in L^\infty(K; \F) : \text{ For each } m \in \Z \backslash\set{0}, \varlimsup_{r \to \infty} \frac{\log n_{V, m}(r)}{\log r} = d}
\end{equation*}
is Baire typical in $L^\infty(K; \F)$. 

Unlike Theorem \ref{thm:1.1}, the proof of Theorem \ref{thm:1.2} requires several new results on resonances in even dimensions, which have analogues in the odd-$d$ case. The first of these are the following sharp upper bounds giving an even-$d$ analogue of \eqref{eq:102} and sharpening the results of \cite{10}, \cite{21}, \cite{22}. 

\begin{thm}\label{thm:1.3}
    Let $d \geq 2$ be even, and $a > 0$. Then for any $V \in L^\infty(\B_a; \C)$, and any $m \in \Z \backslash \set{0}$, there exists $A > 0$ such that 
    \begin{equation*}
        N_{V, m}(r) \leq \frac{c_d a^d}{d}r^d + Ar^{d -1}\log r.
    \end{equation*}
\end{thm}

Part of the proof of this result requires an even-$d$ analogue of \cite[Theorem 3.3]{7} (which sharpened an earlier result, \cite[Theorem 5]{19}) giving a sharp bound on a Fredholm determinant function related to the scattering matrix, see Proposition \ref{prop:2.1}.

The next theorem shows that the characteristic function of a ball achieves the optimal upper bounds in Theorem \ref{thm:1.3}. It sharpens \cite[Theorem 1.2]{6}, and we note that a similar result was proved in odd dimensions by \cite{7}, improving earlier results of \cite{24}, \cite{19}. 

\begin{thm}\label{thm:1.4}
    Let $d \geq 2$ be even and define $V(x) := V_0 \chi_{[0, a]}(\abs{x})$, $x \in \R^d$, where $V_0, a > 0$. Then for each $m \in \Z \backslash \set{0}$ we have 
    \begin{equation*}
        n_{V, m}(r) = c_da^d r^d + O(r^{d - \frac{3}{4} + \varepsilon}) \quad \text{as } r \to \infty, \text{ for any } \varepsilon > 0.
    \end{equation*}
\end{thm}

We prove this theorem by refining the method of \cite{6}, using the uniform asymptotic expansions of Bessel functions. The proof of Theorem \ref{thm:1.4} is ultimately reduced to an application of a technical result used in the proof of its odd-$d$ analogue, which does not depend upon the parity of the dimension. 

Our final result is the even-dimensional analogue of the main result in \cite{7} (which improved earlier results of Christiansen, \cite{1}, \cite{3}) on analytic families of potentials, and improves \cite[Theorem 3.8]{6}. The proof uses results from the theory of several complex variables, and we review the relevant material, including the definition of \emph{pluripolar set}, in Section \ref{sec:4}. Here we note that for $p \geq 1$ an integer, a pluripolar set $E \subset \C^p$ has $\R^{2p}$ Lebesgue measure zero. Moreover, for pluripolar $E \subset \C$, the set $E \cap \R$ has one dimensional Lebesgue measure zero, see e.g., \cite{11}, \cite{14} for these facts. 

\begin{thm}\label{thm:1.5}
    Let $\Omega \subset \C^p$ be an open connected set, and let $\set{V_z}_{z \in \Omega} \subset L^\infty(\B_a; \C)$ be a family analytic in $\Omega$. If for some $m \in \Z \backslash \set{0}, z_m \in \Omega$, and $0 < \delta_m \leq 1$ we have 
    \begin{equation*}
        n_{V_{z_m}, m}(r) = c_d a^d r^d + O(r^{d - \delta_m + \varepsilon}) \quad \text{as } r \to \infty \text{ for any } \varepsilon > 0, \tag{1.4} \label{eq:104}
    \end{equation*}
    then there exists a pluripolar set $E_m \subset \Omega$ such that 
    \begin{equation*}
        n_{V_{z}, m}(r) = c_d a^d r^d + O(r^{d - \frac{\delta_m}{4} + \varepsilon}) \quad \text{as } r \to \infty \text{ for any } \varepsilon > 0, \tag{1.5} \label{eq:105}
    \end{equation*}
    for all $z \in \Omega \backslash E_m$. Moreover, if for \emph{each} $m \in \Z \backslash \set{0}$, there exist $z_m \in \Omega$ and $0 < \delta_m \leq 1$ such that \eqref{eq:104} holds, then there exists a pluripolar set $E \subset \Omega$ such that \eqref{eq:105} holds for every $m \in \Z \backslash \set{0}$ and all $z \in \Omega \backslash E$. 
\end{thm}

We will use this theorem in the proof of Theorem \ref{thm:1.2}, but note that together with Theorem \ref{thm:1.4} it gives an interesting alternative statement that there are many potentials achieving maximal resonance density in even dimensions. It shows that such potentials are generic in the sense of pluripotential theory, see \cite{7} and \cite{Dn} for more discussion on this notion of genericity. Notice also that compared to its odd-$d$ analogue in \cite[Theorem 1.2]{7}, our version removes the requirement that the family be uniformly bounded. An analogue of our proof could be used to remove the restriction there as well. 

From a technical point of view, the primary difference between the proofs of our even-$d$ results and those of their odd-$d$ analogues is that in odd dimensions one may work with a single meromorphic function on $\C$, namely the scattering determinant, whose zeros with finitely many exceptions correspond to resonances. For even dimensions this is no longer true. We must work with a family of functions, each holomorphic in an open half-plane, and whose zeros correspond to resonances on one of the open sheets, $\Lambda_m$. Without analyticity near zero, to count resonances we must use an argument principle applied in a half-annular region, and this breaks some uniformity present in the proofs of the odd-$d$ analogues. However, building off of the approach of \cite{6}, which in turn uses ideas of \cite{21}, \cite{22}, and \cite{F}, we are able to remedy these issues and obtain our results. This modified approach would also work in the odd-$d$ case, but for even dimensions it is essential.

The investigation of the existence and asymptotic number of resonances in potential scattering has a long, rich history. For odd dimensions, we refer to \cite{18}, \cite{7} and references therein, \cite[Ch.'s 2 and 3]{8}, and the survey \cite{26}. As mentioned above, Theorem \ref{thm:1.1} is a continuation of several previous results on resonances for generic potentials, see \cite{1}, \cite{5}, \cite{3}, \cite{7}. 

Until now, corresponding results in even dimensions had only partially been obtained, and as indicated above, this paper fills in some gaps. In addition to the papers already discussed, we remark that further existence and quantitative lower bounds for resonances in even dimensions may be found in \cite{4}, \cite{20}, \cite{15}, and \cite{Chen}. 

Despite some differences in the theories for even and odd dimensions, Theorems \ref{thm:1.1} and \ref{thm:1.2} bring known results on generic potentials to equal footing, and Theorems \ref{thm:1.3} and \ref{thm:1.4} give a sharp upper bound on the asymptotic number of resonances in even dimensions directly analogous to that known in odd dimensions (compare to \cite{24}, \cite{25}, \cite{19}).

The paper is organized as follows: In Section \ref{sec:2} we define the determinant functions from \cite{6} used to obtain our results, relate them to the scattering matrix, and gather some important estimates. Theorem \ref{thm:1.3} will follow as a corollary to these estimates and a Jensen-type formula for a half-annulus. Section \ref{sec:3} gives the proof of Theorem \ref{thm:1.4}, beginning with some preliminaries on Bessel functions. In Section \ref{sec:4} we apply results from the theory of several complex variables to prove Theorem \ref{thm:1.5}. We then provide a modification of the argument of \cite{5} showing how Theorems \ref{thm:1.3}, \ref{thm:1.4}, and \ref{thm:1.5}, and their odd-dimensional analogues combine to give Theorems \ref{thm:1.2} and \ref{thm:1.1}, respectively.

\noindent \textbf{Acknowledgments.} It is a pleasure to thank Tanya Christiansen for suggesting the project and for many helpful discussions throughout the writing of this paper. We also thank Dan Cunningham for logistical help, and Ben Jeffers and the Prison Mathematics Project for handling the typing of the original manuscript.

\section{PRELIMINARIES AND PROOF OF THEOREM \ref{thm:1.3}} \label{sec:2}

In this section we gather some technical preliminaries, mostly following \cite{6}. In any even dimension, the free resolvent $R_0(\lambda):= (-\Delta - \lambda^2)^{-1}$ is holomorphic in the physical plane, $0 < \arg \lambda < \pi$. For any $V \in L_c^\infty(\R^d)$ the resolvent $R_V(\lambda) := (-\Delta + V - \lambda^2)^{-1}$, initially defined in $0 < \arg \lambda < \pi$, $\abs{\lambda} > C$ for large $C$, has a meromorphic continuation to $\Lambda = \exp^{-1}(\C \backslash \set{0})$ as an operator taking $L_c^2(\R^d)$ to $H^2_{\operatorname{loc}}(\R^d)$. Resonances are the poles of this meromorphic continuation.

Next we define the operator $T(\lambda)$ through the Schwartz kernel
\begin{equation*}
    T(\lambda, x, y) := i \pi(2\pi)^{-d} \lambda^{d-2} \int_{\mathbb S^{d-1}} e^{ i\lambda \langle x-y, \omega \rangle} d \omega.  \tag{2.1} \label{eq:201}
\end{equation*}
Then for any $\chi_V \in C_c^\infty(\R^d)$ with $\chi_V V = V$, \cite[Section 2]{6} shows that for each $m \in \Z \backslash \set{0}$, the function 
\begin{equation*}
    f_{V, m}(\lambda) := \det(I - m(I + VR_0(\lambda) \chi_V)^{-1} VT(\lambda)\chi_V) \tag{2.2} \label{eq:202}
\end{equation*}
is well-defined and meromorphic in $\Lambda_0 = \set{0 < \arg \lambda < \pi}$ with at most finitely many poles. With at most finitely many exceptions, $\lambda \in \Lambda_0$ is a zero of $f_{V, m}$ if and only if $e^{i \pi m} \lambda \in \Lambda_m$ is a resonance, with agreement of multiplicities. Moreover, there exists $c_0 > 0$ depending only on $\norm{V}_\infty$ such that $f_{V, m}$ has analytic continuation to a neighborhood of $\set{\lambda \in \Lambda : 0 \leq \arg \lambda \leq \pi, \abs{\lambda} > c_0}$.

By \cite[proof of Lemma 3.4]{6} (see \cite[(24)]{6} and the two inequalities following it), we have the following: There exists $C_{V, m} > 0$ depending only on $\norm{V}_\infty$ and $m \in \Z \backslash\set{0}$, so that 
\begin{equation*}
    \abs{\frac{d}{ds} \log f_{V, m}(s)} \leq C_{V, m}\abs{s}^{d-2} \tag{2.3} \label{eq:203}
\end{equation*}
for $\abs{s} > C_{V, m}$ and either $\arg s = 0$ or $\arg s = \pi$. 

Next we wish to give an analogue of \cite[Theorem 3.3]{7} which refined an earlier bound of \cite{19} and gave an estimate on the scattering determinant in odd dimensions. To state it, we let $\C_+$ denote the open upper half plane and define the functions 
\begin{equation*}
    \rho(z) := \log \frac{1 + \sqrt{1 - z^2}}{z} - \sqrt{1 - z^2}, \quad z \in \overline{\C}_+ \backslash \set{0}
\end{equation*}
and 
\begin{equation*}
    h_d(\theta) := \frac{4}{(d-2)!} \int_0^\infty \frac{\max (- \operatorname{Re} \rho(te^{i \theta}), 0)}{t^{d+1}}dt, \quad \theta \in [0, \pi].
\end{equation*}

For future reference, we also follow Stefanov, \cite{19}, defining 
\begin{equation*}
    c_d := \frac{d}{2\pi} \int_0^\pi h_d(\theta) d\theta = \frac{2d}{\pi(d-2)!} \int_{\operatorname{Im} z > 0} \frac{\max (-\operatorname{Re} \rho(z), 0)}{\abs{z}^{d +2}}dxdy \tag{2.4} \label{eq:204}
\end{equation*}
which is the dimensional constant first mentioned after \eqref{eq:102} and used throughout the paper. 

\begin{prop}\label{prop:2.1}
    For any $V \in L^\infty(\B_a), m \in \Z \backslash \set{0}$, there exist constants $A_{V, m}, C_{V, m}$ depending only on $\norm{V}_\infty$ and $m$ such that 
    \begin{equation*}
        \log \abs{f_{V, m}(re^{i \theta})} \leq h_d(\theta)a^dr^d + A_{V, m}r^{d-1}\log r
    \end{equation*}
    for $r > C_{V, m}$ and $\theta \in [0, \pi]$. 
\end{prop}

We will reduce the proof of Proposition \ref{prop:2.1} to an application of the proof of \cite[Theorem 3.3]{7}. Our first step is to relate the function $f_{V, m}$ to the representation of the scattering matrix used in \cite{7}. 

Let $a > 0, V \in L^\infty(\B_a)$, and select $\chi_i \in C_c^\infty(\R^d, [0, 1]), i=1, 2$ such that for some $R > a$, $\chi_i|_{\B_R} = 1, \chi_2|_{\operatorname{supp} \chi_1} = 1$. We then define the operators $E_\pm(\lambda)\colon L_c^2(\R^d) \to L^2(\mathbb S^{d-1})$ by the Schwartz kernel 
\begin{equation*}
    E_\pm(\lambda, \omega, x) = e^{\pm i \lambda \langle \omega, x \rangle} \quad \text{for } x \in \R^d, \omega \in \mathbb S^{d-1}. \tag{2.5} \label{eq:205}
\end{equation*}
The transpose operators are then defined as the operators $E_\pm^t(\lambda) \colon L^2(\mathbb S^{d-1}) \to L_{\operatorname{loc}}^2(\R^d)$ having this same Schwartz kernel. We then recall the definition of the scattering matrix 
\begin{equation*}
    S_V(\lambda) = I + A_V(\lambda)
\end{equation*}
where 
\begin{equation*}
    A_V(\lambda) = (2\pi)^{-d +1}(2i)^{-1}\lambda^{d-2}E_-(\lambda)[\Delta, \chi_1] R_V(\lambda)[\Delta, \chi_2]E_+^t(\lambda)
\end{equation*}
(see \cite[Theorem 4.26]{8} which shows that this representation holds for more general black box operators). For $d$ even, $S_V(\lambda)$ and $A_V(\lambda)$ are meromorphic in $\lambda \in \Lambda$ as operators on $L^2(\mathbb S^{d-1})$. 

\begin{lm}\label{lem:2.2}
    For any $V \in L^\infty(\B_a)$, and $m \in \Z \backslash \set{0}$, we may write 
    \begin{equation*}
        f_{V, m}(\lambda) = \det(I + m A_V(\lambda)), \quad \lambda \in \Lambda_0.
    \end{equation*}
\end{lm}
\begin{proof}
    First notice that from \eqref{eq:201} and \eqref{eq:205}
    \begin{equation*}
        \chi_VT(\lambda) \chi_V = -(2\pi)^{-d + 1}(2i)^{-1} \lambda^{d-2} \chi_V E_+^t(\lambda) E_-(\lambda) \chi_V.
    \end{equation*}
Hence, by cyclicity of the determinant (e.g., \cite[B.5.18]{8}), it suffices to show that as operators on $L^2(\mathbb S^{d-1})$, we have 
\begin{equation*}
    E_-(\lambda)(I + VR_0(\lambda)\chi_V)^{-1}VE_+^t(\lambda) = E_-(\lambda)[\Delta, \chi_1]R_V(\lambda)[\Delta, \chi_2]E_+^t(\lambda), \quad \lambda \in \Lambda_0.
\end{equation*}
But this follows directly from the proof of \cite[Theorem 3.44]{8}, which makes no use of the parity of the dimension, provided we use the boundary pairing formula in \cite[Theorem 4.23]{8} and at the end of the proof we analytically continue the identity to a neighborhood of $\set{\lambda \in \Lambda : 0 \leq \arg \lambda \leq \pi, \abs{\lambda} > c_0}$ rather than to $\C$. 
\end{proof}

The remainder of the proof of Proposition \ref{prop:2.1} now follows essentially as in \cite[Section 3]{7}. For $\lambda \in \Lambda$ we define $\overline{\lambda} := \abs{\lambda}e^{-i \arg \lambda}$. Then it suffices to prove the bound for $\theta \in [0, \frac{\pi}{2}]$, since using $T(- \overline{\lambda}) = -\overline{T(\lambda)}$, we have 
\begin{equation*}
    \overline{f_{\overline V, -m}(\lambda)} = f_{V, m}(-\overline{\lambda}), \quad \lambda \in \Lambda_0. \tag{2.6} \label{eq:206}
\end{equation*}
In view of Lemma \ref{lem:2.2}, all that remains in order to apply the proof of \cite[Theorem 3.3]{7} is to note that for any $\rho \in C_c^\infty(\R^d)$, standard estimates for solutions to elliptic equations provides the bound 
\begin{equation*}
    \norm{\rho R_V(re^{i \theta}) \rho}_{H^{-1} \to H^1} = O(r^2).
\end{equation*}
It is straightforward to check that the remainder of the proof of Proposition \ref{prop:2.1} follows by applying \cite[proof of Theorem 3.3]{7} to the representation of $f_{V, m}$ given in Lemma \ref{lem:2.2}. 

We are now ready to use the preliminaries established above to prove a theorem which implies Theorem \ref{thm:1.3}. An odd dimensional version of the first part of this theorem was obtained by Christiansen, \cite[(3.2)]{3}, and see also Stefanov, \cite[Proposition 2]{19}, and Dinh-Vu, \cite[Proposition 3.2]{7}. We state it for an analytic family of potentials because it will be useful to us later in Section \ref{sec:4}, where we use it to prove Theorem \ref{thm:1.5}. 

\begin{thm}\label{thm:2.3}
    Let $\Omega \subset \C^p$ be an open connected set, and let $\set{V_z}_{z \in \Omega} \subset L^\infty(\B_a; \C)$ be a family analytic in $\Omega$. Then for any $\Omega' \subset\subset \Omega$ and fixed $m \in \Z \backslash \set{0}$, there exists a constant $A_{\Omega', m}$ such that 
    \begin{equation*}
        \abs{N_{V_z, m}(r) - \frac{1}{2\pi} \int_0^\pi \log \abs{f_{V_z, m}(re^{i \theta})} d\theta} \leq A_{\Omega', m} r^{d-1}, \quad z \in \Omega' \tag{2.7} \label{eq:207}
    \end{equation*}
    for large enough $r$. In particular, we have for some $A_{\Omega', m}'$, 
    \begin{equation*}
        N_{V_z, m}(r) \leq \frac{c_da^d}{d}r^d + A'_{\Omega', m}r^{d-1}\log r \tag{2.8} \label{eq:208}
    \end{equation*}
    uniformly for $z \in \Omega'$. 
\end{thm}

\begin{proof}
    Let $\Omega''$ be open, connected, and $\Omega' \subset \subset \Omega'' \subset \subset \Omega$. Then $\norm{V_z}_\infty$ is uniformly bounded for $z \in \Omega''$, and by \eqref{eq:203} and Proposition \ref{prop:2.1} we can find constants $A, C > 0$ such that 
    \begin{align*}
        \abs{\frac{d}{ds} \log f_{V_z, m}(s)} \leq&\; C \abs{s}^{d-2}, \quad \abs{s} > C,\; \arg s = 0 \text{ or } \pi\tag{2.9} \label{eq:209}\\
        \log \abs{f_{V_z, m}(re^{i \theta})} \leq&\; h_d(\theta)a^dr^d + Ar^{d-1}\log r, \quad r > C, \theta \in [0, \pi] \tag{2.10} \label{eq:210}
    \end{align*}
    uniformly for $z \in \Omega''$. 

    Let $z_0 \in \overline{\Omega'}$. Then $\exists R_{z_0} > C$ such that $f_{V_{z_0}, m}(\cdot)$ has no zeros on $\abs{\lambda} = R_{z_0}$. By continuity of zeros of analytic functions, there is an open set $\Omega_{z_0} \subset \Omega''$ so that $f_{V_z, m}(R_{z_0}e^{i \theta}) \neq 0$ for all $z \in \Omega_{z_0}$, $\theta \in [0, \pi]$. Thus, $\log \abs{f_{V_z, m}(Re^{i\theta}) }$ and its derivatives are continuous for such $z, \theta$, and $R$ near $R_{z_0}$. We therefore apply the formula \cite[(18)]{6} to $f_{V_z, m}$ with $z \in \Omega_{z_0}$ to find 
    \begin{align*}
       & \int_{R_{z_0}}^r \frac{n_{V_z, m}(t)}{t}dt = \frac{1}{2\pi} \operatorname{Im}\left(\int_{R_{z_0}}^r t^{-1} \int_{-t}^{-R_{z_0}} \frac{\frac{d}{ds} f_{V_z, m}(s)}{f_{V_z, m}(s)}dsdt  + \int_{R_{z_{0}}}^r t^{-1} \int_{R_{z_0}}^t \frac{\frac{d}{ds}f_{V_z, m}(s)}{f_{V_z, m}(s)}dsdt \right)\\ 
       & +\frac{1}{2\pi} \int_0^\pi \log \abs{f_{V_z, m}(re^{i \theta})}d\theta - \frac{1}{2\pi} \int_0^\pi \log \abs{f_{V_z, m}(R_{z_0}e^{i\theta})} d\theta \\
       &- \frac{1}{2\pi} R_{z_0} \log\left( \frac{r}{R_{z_0}} \right) \int_0^\pi \frac{d}{dR} \log \abs{f_{V_z, m}(Re^{i\theta})}\bigg\rvert_{R = R_{z_0}} d\theta\tag{2.11} \label{eq:211}
    \end{align*}

    It follows from \eqref{eq:209} and \eqref{eq:211} that we can find a constant $A_{z_0, m} > 0$ such that
    \begin{equation*}
        \abs{\int_{R_{z_0}}^r \frac{n_{V_z, m}(t)}{t}dt - \frac{1}{2\pi} \int_0^\pi \log \abs{f_{V_z, m}(re^{i\theta})}d\theta} \leq A_{z_0, m}r^{d-1}, \quad z \in \Omega_{z_0}, r > A_{z_0, m}. 
    \end{equation*}
    Again by continuity of zeros for analytic functions, we see that $n_{V_z, m}(R_{z_0})$ and hence $\int_0^{R_{z_0}} t^{-1}n_{V_z, m}(t)dt$ are uniformly bounded for $z \in \Omega_{z_0}$. Thus, for a constant $A_{z_0, m}' > A_{z_0, m}$, 
    \begin{equation*}
        \abs{N_{V_z, m}(r) - \frac{1}{2\pi} \int_0^\pi \log \abs{f_{V_z, m}(re^{i\theta})} d\theta} \leq A_{z_0, m}'r^{d-1}, \quad z \in \Omega_{z_0}, r > A_{z_0, m}'.
    \end{equation*}
    By construction, $\overline{\Omega'} \subset \bigcup_{z_0 \in \overline{\Omega'}} \Omega_{z_0}$ so $\set{\Omega_{z_0}}_{z_0 \in \overline{\Omega'}}$ forms a cover of $\overline{\Omega'}$. Thus since $\overline{\Omega'}$ is compact, we have $\overline{\Omega'} \subset \bigcup_{i=1}^N \Omega_{z_i}$ for a \emph{finite} union. The claim \eqref{eq:207} clearly follows from this. 

    The claim \eqref{eq:208} follows from \eqref{eq:207} and \eqref{eq:210} once we recall \eqref{eq:204}.
\end{proof}

Clearly, Theorem \ref{thm:1.3} follows from \eqref{eq:208} by taking a family $V_z \equiv V$. 

\section{MAXIMAL RESONANCE DENSITY FOR A RADIAL STEP POTENTIAL IN EVEN DIMENSIONS}\label{sec:3}

\subsection{Preliminaries and uniform expansions of Bessel functions}\label{sec:3.1}
We prove Theorem \ref{thm:1.4} in Subsection \ref{sec:3.2}, and we lay the groundwork here by recalling some material from \cite{6}, as well as the uniform asymptotic expansions of Bessel functions obtained by Olver, \cite{13}. 

Without loss, we will normalize so that $a =1$. Here we define $V(x) := V_0 \chi_{[0, 1]}(\abs{x}), V_0 > 0$, as in Theorem \ref{thm:1.4}. We also set $\Sigma(\lambda) := (\lambda^2 - V_0)^{\frac{1}{2}}$ with the square root having branch cuts $\left(-\infty, -V_0^{\frac{1}{2}}\right] \cup \left[V_0^{\frac{1}{2}}, \infty \right)$. For $\ell = 0, 1, 2, \ldots$ and $\nu = \ell + (d-2)/2$, it is shown in \cite[Section 5]{6} that $\lambda_0 \in \Lambda_0$ is a zero of the function 
\begin{equation*}
    F_m^{(\nu)}(\lambda) := \Sigma(\lambda)J_\nu'\left( \Sigma(\lambda) \right) H_\nu^{(1)}\left( e^{im \pi} \lambda \right) - e^{im \pi} \lambda J_\nu \left( \Sigma(\lambda)\right) H_\nu^{(1)'}\left(e^{im \pi} \lambda \right), \quad \lambda \in \Lambda_0, \tag{3.1} \label{eq:301}
\end{equation*}
where $J_\nu, H_\nu^{(1)}$ are the Bessel, Hankel functions respectively (see \cite{13}, \cite{12}), if and only if $\lambda_1 = e^{im \pi} \lambda_0 \in \Lambda_m$ is a resonance. Moreover, the multiplicity of the resonance $\lambda_1$ is equal to $M(\ell)$, the dimension of the space of spherical harmonics on $\mathbb S^{d-1}$ with eigenvalue $\ell(\ell + d - 2)$. 

To study the zeros of the functions in \eqref{eq:301}, we follow \cite{6}, putting $z = \lambda/\nu, \tilde{z} = \tilde{z}(z) := (z^2 - V_0/\nu^2)^{\frac{1}{2}}$ and using \cite[(55)-(57)]{6} to write 
\begin{equation*}
    F_m^{(\nu)}(\nu z) = (-1)^{m\nu}\left[ F_0^{(\nu)}(\nu z) - 2mG_0^{(\nu)} (\nu z)\right], \quad z \in \Lambda_0 \tag{3.2} \label{eq:302}
\end{equation*}
where 
\begin{equation*}
    F_0^{(\nu)}(\nu z) = \nu \tilde{z} J_\nu'(\nu \tilde{z})H_\nu^{(1)}(\nu z) - \nu zJ_\nu(\nu \tilde{z})H_\nu^{(1)'}(\nu z) \tag{3.3} \label{eq:303}
\end{equation*}
and 
\begin{equation*}
    G_0^{(\nu)}(\nu z) = \nu \tilde{z}J_\nu'(\nu \tilde{z})J_\nu(\nu z) - \nu z J_\nu(\nu \tilde{z})J_\nu'(\nu z). \tag{3.4} \label{eq:304}
\end{equation*}

Next we record the uniform asymptotic expansions of Bessel and Hankel functions from \cite{13}. To state the expansions, we first recall the definition of $\rho(z)$ and define the function $\zeta(z)$ through 
\begin{equation*}
    \rho(z) = \frac{2}{3}\zeta(z)^{3/2} = \log \frac{1 + \sqrt{1 - z^2}}{z} - \sqrt{1-z^2}, \quad z \in \overline{\C}_+ \backslash\set{0}. \tag{3.5} \label{eq:305}
\end{equation*}
Notice that 
\begin{equation*}
    \rho'(z) = \zeta(z)^{1/2}\zeta'(z) = - \frac{\sqrt{1 - z^2}}{z}. \tag{3.6} \label{eq:306}
\end{equation*}
For further properties of these functions we refer to \cite{13}, \cite{12}. We mention here that $\abs{z} \to \infty$ through $\C_+$ if and only if $\abs{\zeta} \to \infty$ through $\set{-\pi < \arg \zeta < - \frac{\pi}{3}}$ and in that case
\begin{equation*}
    z = \frac{2}{3}(-\zeta)^{2/3} + \frac{\pi}{2} + O(\abs{\zeta}^{-1}), \tag{3.7} \label{eq:307}
\end{equation*}
(see \cite[(4.14)]{13}).

We also let $Ai(z)$ denote the Airy function, and set $Ai_-(z) := Ai(e^{2\pi i/3}z)$. Note that $Ai''(z) = zAi(z)$. Finally, define the functions 
\begin{equation*}
    \phi(\zeta) := \left( \frac{4\zeta}{1 - z^2} \right)^{1/4},  \quad \psi(\zeta) = \frac{2}{z\phi(\zeta)}. 
\end{equation*}
Then Olver, \cite{13}, proved that for any $\delta > 0$, the following expansions hold as $\nu \to \infty$ uniformly in $0 \leq \arg z \leq \pi - \delta$: 
\begin{align*}
    J_\nu(\nu z) \sim &\;\phi(\zeta) \left(\frac{Ai(\nu^{2/3}\zeta)}{\nu^{1/3}} \sum_{j=0}^\infty\frac{A_j(\zeta)}{\nu^{2j}} + \frac{Ai'(\nu^{2/3}\zeta)}{\nu^{5/3}} \sum_{j=0}^\infty \frac{B_j(\zeta)}{\nu^{2j}} \right), \tag{3.8} \label{eq:308}\\
    J_\nu'(\nu z) \sim &\; -\psi(\zeta) \left(\frac{Ai(\nu^{2/3}\zeta)}{\nu^{4/3}} \sum_{j=0}^\infty\frac{C_j(\zeta)}{\nu^{2j}} + \frac{Ai'(\nu^{2/3}\zeta)}{\nu^{2/3}} \sum_{j=0}^\infty \frac{D_j(\zeta)}{\nu^{2j}} \right), \tag{3.9} \label{eq:309}\\
    H_\nu^{(1)}(\nu z) \sim &\; 2e^{-i\pi/3}\phi(\zeta) \left(\frac{Ai_-(\nu^{2/3}\zeta)}{\nu^{1/3}} \sum_{j=0}^\infty\frac{A_j(\zeta)}{\nu^{2j}} + \frac{Ai_-'(\nu^{2/3}\zeta)}{\nu^{5/3}} \sum_{j=0}^\infty \frac{B_j(\zeta)}{\nu^{2j}} \right), \tag{3.10} \label{eq:310}\\
    H_\nu^{(1)'}(\nu z) \sim &\; -2e^{-i\pi/3}\psi(\zeta) \left(\frac{Ai_-(\nu^{2/3}\zeta)}{\nu^{4/3}} \sum_{j=0}^\infty\frac{C_j(\zeta)}{\nu^{2j}} + \frac{Ai_-'(\nu^{2/3}\zeta)}{\nu^{2/3}} \sum_{j=0}^\infty \frac{D_j(\zeta)}{\nu^{2j}} \right), \tag{3.11} \label{eq:311}\\
\end{align*}

In the above expansions, $A_0(\zeta) \equiv 1$ and all other coefficients are determined recursively, see \cite[(4.17), (6.10), (6.11)]{13}. All coefficients are holomorphic (as functions of $z$) in $0 \leq \arg z \leq \pi-\delta$; see \cite{13} for more information. Here only the following additional properties of these coefficients will be needed. They follow directly from \cite[(6.6), (6.12)]{13} and \eqref{eq:307}. 

\begin{prop}\label{prop:3.1}
    The following bounds hold uniformly in $0 \leq \arg z \leq \pi - \delta$:
    \begin{enumerate}[a)]
        \item $A_0(\zeta) \equiv 1, A_1(\zeta) = O(\abs{\zeta}^{-3}), A_2(\zeta) = O(\abs{\zeta}^{-6})$,
        
        \item $B_0(\zeta) = O(\abs{\zeta}^{-2}), B_1(\zeta) = O(\abs{\zeta}^{-5})$,

        \item $C_0(\zeta) = O(\abs{\zeta}^{-1}), C_1(\zeta) = O(\abs{\zeta}^{-4})$,

        \item $D_0(\zeta) \equiv 1, D_1(\zeta) = O(\abs{\zeta}^{-3}), D_2(\zeta) = O(\abs{\zeta}^{-6})$.
    \end{enumerate}
    Moreover,
    \begin{enumerate}[a')]
        \item $\frac{d}{dz}A_1(\zeta) = O(\abs{\zeta}^{-9/2})$

        \item $\frac{d}{dz}B_0(\zeta) = O(\abs{\zeta}^{-7/2})$

        \item $\frac{d}{dz}C_0(\zeta) = O(\abs{\zeta}^{-5/2})$

        \item $\frac{d}{dz}D_1(\zeta) = O(\abs{\zeta}^{-9/2})$.
    \end{enumerate}
\end{prop}

The Airy function also has asymptotic expansions. If $\xi = \frac{2}{3}w^{3/2}$, then 
\begin{align*}
    Ai(w) \sim&\; \frac{e^{-\xi}}{2\pi^{1/2}w^{1/4}}\left(1 + \sum_{s=1}^\infty \frac{u_s}{(-\xi)^s} \right), \tag{3.12} \label{eq:312}\\
    Ai'(w) \sim&\; \frac{w^{1/4}e^{-\xi}}{2\pi^{1/2}}\left(-1 + \sum_{s=1}^\infty \frac{u_s'}{(-\xi)^s} \right), \tag{3.13} \label{eq:313}
\end{align*}
as $\abs{w} \to \infty$ in $\abs{\arg w} \leq \pi - \delta$ for any $\delta > 0$. Similarly, 
\begin{align*}
    Ai(-w) \sim&\; \frac{1}{\pi^{1/2} w^{1/4}} \left[\cos \left( \xi- \frac{\pi}{4} \right)\left(1 + \sum_{s=1}^\infty \frac{a_s}{\xi^{2s}}\right) + \sin \left( \xi - \frac{\pi}{4} \right)\sum_{s=0}^\infty \frac{b_s}{\xi^{2s+1}}  \right], \tag{3.14} \label{eq:314}\\
    Ai'(-w) \sim&\; \frac{w^{1/4}}{\pi^{1/2}} \left[\sin \left( \xi- \frac{\pi}{4} \right)\left(1 + \sum_{s=1}^\infty \frac{a_s'}{\xi^{2s}}\right) + \cos \left( \xi - \frac{\pi}{4} \right)\sum_{s=0}^\infty \frac{b_s'}{\xi^{2s+1}}  \right], \tag{3.15} \label{eq:315}\\
\end{align*}
as $\abs{w} \to \infty$ in $\abs{\arg w} \leq \frac{2\pi}{3} - \delta$ for any $\delta > 0$. Here $u_s, u_s', a_s, b_s, a_s', b_s'$ are numerical constants that are possible to determine explicitly (see \cite[Appendix]{13}). Since all of the above expansions are differentiable, we remark that it is not difficult to check that 
\begin{equation}
    b_0 + b_0' = \frac{1}{6} \quad \text{and} \quad u_1 + u_1' = \frac{1}{6}. \tag{3.16} \label{eq:316}
\end{equation}

Finally, since \eqref{eq:303} and \eqref{eq:304} involve $\tilde{z}(z) = (z^2 - V_0/\nu^2)^{1/2}$ we will need to compare functions of $\tilde{z}$ with functions of $z$. To do this, we first fix small constants $c, \delta > 0$ and define the set 
\begin{equation*}
    U_v^\delta := \set{z \in \C_+ : \arg z \in (0, \pi -\delta), \operatorname{Re}\rho(z) < \frac{-\log (c\nu)}{2\nu}}. 
\end{equation*}
Notice that in this region $\abs{\rho} \gtrsim \nu^{-1}\log \nu$ and $\nu^2 \abs{1 - z^2} \gtrsim \nu^{4/3} (\log \nu)^{4/3}$. Using Taylor's formula, \eqref{eq:305}, \eqref{eq:306}, \eqref{eq:307}, we obtain 
\begin{align*}
    \left(\frac{1-\tilde{z}^2}{1-z^2} \right)^{\pm \frac{1}{4}} =&\; 1 \pm \frac{V_0}{4\nu^2(1-z^2)}(1 + o(1)), \tag{3.17} \label{eq:317}\\
    \rho(\tilde{z}) =&\; \rho(z) + O\left(\frac{1}{\nu^2} \right), \tag{3.18}\label{eq:318}\\
    \zeta(\tilde{z})^{\pm \frac{1}{4}} =&\; \zeta(z)^{\pm \frac{1}{4}}\mp \frac{V_0}{8\nu^2z} \cdot \frac{\zeta'(z)}{\zeta(z)^{1\mp \frac{1}{4}}} + \frac{1}{\nu^2(1-z^2)}o(1), \tag{3.19} \label{eq:319}\\
    Ai(\nu^{2/3} \zeta(\tilde{z})) =&\;Ai(\nu^{2/3}\zeta) - \frac{V_0}{2\nu^{4/3}z}\cdot \zeta'Ai'(\nu^{2/3}\zeta) + \frac{V_0^2}{4\nu^2z^2}(\zeta')^2 \zeta Ai(\nu^{2/3}\zeta) + \frac{1}{\nu^2(1-z^2)}e^{-\nu\rho}o(1) \tag{3.20} \label{eq:320}\\
    Ai'(\nu^{2/3}\zeta(\tilde{z})) =&\; Ai'(\nu^{2/3}\zeta)  - \frac{V_0}{2\nu^{2/3}z}\zeta' \zeta Ai(\nu^{2/3}\zeta) + \frac{V_0^2}{4\nu^2z^2}(\zeta')^2\zeta Ai'(\nu^{2/3}\zeta) \\
    +&\; \frac{1}{\nu^2(1-z^2)}e^{-\nu \rho}o(1) \tag{3.21} \label{eq:321}
\end{align*}
as $\nu \to \infty$, uniformly for $z \in U_\nu^\delta$. Similarly, using Proposition \ref{prop:3.1}, 
\begin{align*}
    A_1(\zeta(\tilde z)) =&\; A_1(\zeta) + \frac{1}{\nu^2(1-z^2)}O(1), \tag{3.22} \label{eq:322}\\
    B_0(\zeta(\tilde z)) =&\; B_0(\zeta) + \frac{1}{\nu^2(1-z^2)}O(1), \tag{3.23} \label{eq:323}\\
    C_0(\zeta(\tilde z)) =&\; C_0(\zeta) + \frac{1}{\nu^2(1-z^2)}O(1), \tag{3.24} \label{eq:324}\\
    D_1(\zeta(\tilde z)) =&\; D_1(\zeta) + \frac{1}{\nu^2(1-z^2)}O(1), \tag{3.25} \label{eq:325}
\end{align*}
as $\nu \to \infty$, uniformly for $z \in U_\nu^\delta$. 

\subsection{Proof of Theorem \ref{thm:1.4}}\label{sec:3.2}
We shall prove the following proposition on the asymptotic behavior of the functions $F_0^{(\nu)}, G_0^{(\nu)}$ defined in \eqref{eq:303}, \eqref{eq:304} respectively. 

\begin{prop}\label{prop:3.2}
    The functions $F_0^{(\nu)}, G_0^{(\nu)}$ satisfy 
    \begin{align*}
        F_0^{(\nu)}(\nu z) =&\; - \frac{2i}{\pi}(1 + o(1)) \tag{3.26} \label{eq:326}\\
        G_0^{(\nu)}(\nu z) =&\; O\left( \frac{1}{\nu} \right) - \frac{V_0}{4\pi \nu^2(1 - z^2)}e^{-2\nu\rho}(1 + o(1)) \tag{3.27} \label{eq:327}\\
    \end{align*}
    as $\nu \to \infty$ uniformly for $z \in U_\nu^\delta$. 
\end{prop}

Before giving the proof, let us show how this proposition gives Theorem \ref{thm:1.4}. Fix $m \in \Z \backslash\set{0}$. We explained in Subsection \ref{sec:3.1} that $n_{V, m}(r)$ is equal to $\sum_{\ell = 0}^\infty M(\ell)n_{m, \ell}(r)$ where $\nu = \ell + (d-2)/2$ and $n_{m, \ell}(r)$ is the number of zeros of $F_m^{(\nu)}$ in $\set{0 < \arg \lambda < \pi, \abs{\lambda} < r}$. To study the zeros of $F_m^{(\nu)}$, we use \eqref{eq:302} and Proposition \ref{prop:3.2} to write 
\begin{equation*}
    \frac{2\pi(-1)^{m \nu}}{mV_0}F_m^{(\nu)}(\nu z) = \frac{e^{-2\nu \rho}}{\nu^2(1-z^2)}(1 + \varepsilon_{m, \nu}(z)) - \frac{4}{i m V_0}(1 + \varepsilon_{m, \nu}'(z)), \quad z\in U_\nu^\delta, \tag{3.28} \label{eq:328}
\end{equation*}
where $\varepsilon_{m, \nu}(z), \varepsilon_{m, \nu}'(z)$ converge uniformly to zero as $\nu \to \infty$. 

Zeros of equations of the form in \eqref{eq:328} are studied in \cite[Section 4]{7} (the fact that our equation \eqref{eq:328} holds in a slightly smaller region than in \cite[(4.1)]{7} makes no difference as it is easy to check that all of the relevant analysis in \cite[pgs. 195-199]{7} takes place only in a region contained in $U_\nu^\delta$). If we let $n_{m, \ell}^1(r)$ and $n_{m, \ell}^2(r)$ denote the number of zeros of $F_m^{(\nu)}$ in $\set{\operatorname{Re} \lambda \geq 0, \operatorname{Im} \lambda \geq 0, \abs{\lambda} < r}$ and $\set{\operatorname{Re} \lambda \leq 0, \operatorname{Im} \lambda \geq 0, \abs{\lambda} < r}$ respectively, then it is the content of \cite[pgs. 195-199]{7} that 
\begin{equation*}
    \sum_{\ell=0}^\infty M(\ell)n_{m, \ell}^1(r) = \sum_{0 \leq \ell \leq 2r} M(\ell)n_{m, \ell}^1(r) = \frac{c_d}{2}r^d + O(r^{d-\frac{3}{4} + \varepsilon}) \text{ as } r \to \infty, \text{ for any } \varepsilon > 0.
\end{equation*}

Next we use \eqref{eq:206} and the characterization of resonances as the zeros of the functions $f_{V,m}$ defined in \eqref{eq:202} to see that $n_{m, \ell}^2(r) = n_{-m, \ell}^1(r)$, and hence applying the above reasoning with $m$ replaced by $-m$, we are able to prove that 
\begin{equation*}
    \sum_{\ell = 0}^\infty M(\ell)n_{m, \ell}^2(r) = \frac{c_d}{2}r^d + O(r^{d - \frac{3}{4} + \varepsilon}) \quad \text{as $r \to \infty$, for any } \varepsilon > 0.
\end{equation*}
Thus, 
\begin{equation*}
    n_{V, m}(r)  = \sum_{\ell = 0}^\infty M(\ell)n_{m, \ell}(r) = c_dr^d + O(r^{d - \frac{3}{4} + \varepsilon}) \quad \text{as $r \to \infty$, for any } \varepsilon > 0,
\end{equation*}
concluding the proof of Theorem \ref{thm:1.4}. 

The remainder of this section gives the proof of Proposition \ref{prop:3.2}. Since we need to distinguish the regions of validity of the Airy function expansions, we begin by fixing a small $\varepsilon > 0$, and set 
\begin{align*}
    U_1 :=&\; U_\nu^\delta \cap \set{z : \arg \zeta(z) \in (-\pi, -\pi + 2\varepsilon)},\\
    U_2 := &\; U_\nu^\delta \cap \set{z \colon \arg \zeta(z) \in (-\pi + \varepsilon, -\frac{\pi}{3})}.
\end{align*}
Then $U_1 \cup U_2 = U_\nu^\delta$ and it suffices to show that \eqref{eq:326}, \eqref{eq:327} hold uniformly in both $U_1$ and $U_2$.  

We begin by showing that \eqref{eq:326} holds in $U_1$. For $\arg \zeta \in (-\pi, -\pi + 2\varepsilon)$, we have $\arg(-\zeta) \in (0, 2\varepsilon)$. Thus we may apply \eqref{eq:308}, \eqref{eq:309}, \eqref{eq:314}, \eqref{eq:315}, and use the bounds of Proposition \ref{prop:3.1} to arrange 
\begin{align*}
    J_\nu(\nu z) =&\; \frac{\phi(\zeta)}{2\pi^{1/2}\nu^{1/2}\zeta^{1/4}}e^{-\nu \rho}(1 + o(1)) \quad  \text{as } \nu \to \infty, z \in U_1, \tag{3.29} \label{eq:329}\\
    J_\nu'(\nu z) =&\; \frac{-\psi(\zeta)\zeta^{1/4}}{2\pi^{1/2}\nu^{1/2}}e^{-\nu \rho}(1 + o(1)) \quad  \text{as } \nu \to \infty, z \in U_1. \tag{3.30} \label{eq:330}\\
\end{align*}
Next, we use that $\arg(e^{2\pi i/3} \zeta)\in (-\frac{\pi}{3}, -\frac{\pi}{3} + 2\varepsilon)$, along with \eqref{eq:310}, \eqref{eq:311}, \eqref{eq:312}, \eqref{eq:313} and the bounds of Proposition \ref{prop:3.1} again to get 
\begin{align*}
    H_{\nu}^{(1)}(\nu z) =&\; \frac{i \phi(\zeta)}{\pi^{1/2}\nu^{1/2}\zeta^{1/4}}e^{\nu \rho}(1 + o(1))\quad \text{as } \nu \to \infty, z\in U_1 \tag{3.31} \label{eq:331}\\
    H_{\nu}^{(1)'}(\nu z) =&\; \frac{-i \psi(\zeta)\zeta^{1/4}}{\pi^{1/2}\nu^{1/2}}e^{\nu \rho}(1 + o(1))\quad \text{as } \nu \to \infty, z\in U_1 \tag{3.32} \label{eq:332}
\end{align*}
Recalling \eqref{eq:303}, we find 
\begin{equation*}
    F_0^{\nu}(\nu z) = -\frac{i}{\pi} \left(\frac{1 - \tilde{z}^2}{1 - z^2} \right)^{\frac{1}{4}}e^{\nu(\rho - \tilde{\rho})}(1 + o(1)) - \frac{i}{\pi} \left(\frac{1 - z^2}{1 - \tilde{z}^2} \right)^{\frac{1}{4}}e^{\nu(\rho - \tilde{\rho})}(1 + o(1))
\end{equation*}
as $\nu \to \infty$, uniformly for $z \in U_1$ (here, to ease notation, we have written $\tilde{\rho} = \rho(\tilde{z})$, and will similarly write $\tilde{\zeta} = \zeta(\tilde{z})$ below). Using \eqref{eq:317}, we have 
\begin{equation*}
    \left(\frac{1-\tilde{z}^2}{1 - z^2} \right)^{\pm \frac{1}{4}} = 1 + o(1) \quad \text{as } \nu \to \infty, z \in U_1,
\end{equation*}
and using \eqref{eq:318}, we have 
\begin{equation*}
    e^{\nu(\rho - \tilde \rho)} = 1+o(1) \quad \text{as } \nu \to \infty, z\in U_1.
\end{equation*}
Thus \eqref{eq:326} holds in $U_1$. 

To prove that \eqref{eq:326} also holds in $U_2$, first note that the proof that \eqref{eq:331}, \eqref{eq:332} hold in $U_2$ follows without change. Using \eqref{eq:308}, \eqref{eq:309}, \eqref{eq:312}, \eqref{eq:313}, it's easy to check that \eqref{eq:329} and \eqref{eq:330} also still hold in $U_2$. The proof that \eqref{eq:326} holds in $U_2$ then follows as above, and we see that \eqref{eq:326} holds in all of $U_\nu^\delta$. 

We now turn to the proof of \eqref{eq:327}, and before distinguishing between $U_1$ and $U_2$, we begin with some preliminary considerations. Using \eqref{eq:308} and \eqref{eq:309} we may write 
\begin{align*}
    J_\nu(\nu z) =&\; \frac{\phi(\zeta)}{\nu^{1/2}\zeta^{1/4}}U(\zeta, \nu),\\
    J_\nu'(\nu z) =&\; -\frac{\psi(\zeta)\zeta^{1/4}}{\nu^{1/2}}W(\zeta, \nu) \tag{3.33} \label{eq:333}
\end{align*}
for functions $U, W$ satisfying 
\begin{equation*}\tag{3.34} \label{eq:334}
    \begin{aligned}
        U(\zeta, \nu) \sim &\; \zeta^{1/4}\nu^{1/6}Ai(\nu^{2/3}\zeta) \sum_{j=0}^\infty \frac{A_j(\zeta)}{\nu^{2j}} + \frac{\zeta^{1/4} Ai'(\nu^{2/3} \zeta)}{\nu^{7/6}} \sum_{j=0}^\infty \frac{B_j(\zeta)}{\nu^{2j}},\\
    W(\zeta, \nu) \sim &\; \frac{\zeta^{-1/4}Ai(\nu^{2/3}\zeta)}{\nu^{5/6}} \sum_{j=0}^\infty \frac{C_j(\zeta)}{\nu^{2j}} + \frac{\zeta^{-1/4} Ai'(\nu^{2/3} \zeta)}{\nu^{1/6}} \sum_{j=0}^\infty \frac{D_j(\zeta)}{\nu^{2j}}. 
    \end{aligned}
\end{equation*}

Thus from \eqref{eq:304} and \eqref{eq:333}, 
\begin{equation*}
    G_0^{(\nu)}(\nu z) = -2\left(\frac{1 - z^2}{1 - \tilde z^2} \right)^{1/4}U(\zeta, \nu)W(\tilde \zeta, \nu) + 2\left(\frac{1 - \tilde z^2}{1-z^2} \right)^{1/4}U(\tilde \zeta, \nu)W(\zeta, \nu) \tag{3.35} \label{eq:335}
\end{equation*}
Moreover, using that the leading behavior in \eqref{eq:312} and \eqref{eq:314} agree, as does that in \eqref{eq:313} and \eqref{eq:315}, we obtain from \eqref{eq:334} that 
\begin{equation*}\tag{3.36}\label{eq:336}
    \begin{aligned}
        U(\zeta, \nu)W(\tilde \zeta, \nu) =&\; - \frac{1}{4\pi} e^{-2\nu \rho}(1 + o(1)), \\ 
    U(\tilde \zeta, \nu)W(\zeta, \nu) =&\; - \frac{1}{4\pi} e^{-2\nu \rho}(1 + o(1)),
    \end{aligned}
\end{equation*}
as $\nu \to \infty$, uniformly for $z \in U_\nu^\delta$. From \eqref{eq:335}, \eqref{eq:336} and \eqref{eq:317}, we have 
\begin{equation*}
    G_0^{(\nu)} (\nu z) = -2\left[U(\zeta, \nu)W(\tilde \zeta, \nu) - U(\tilde \zeta, \nu)W(\zeta, \nu)\right] - \frac{V_0}{4\pi\nu^2(1 - z^2)}e^{-2\nu \rho}(1 + o(1)), \tag{3.37} \label{eq:337}
\end{equation*}
as $\nu \to \infty$, uniformly for $z \in U_\nu^\delta$. Thus, we just need to prove that 
\begin{equation*}
    U(\zeta, \nu)W(\tilde \zeta, \nu) - U(\tilde \zeta, \nu)W(\zeta, \nu) = O\left( \frac{1}{\nu} \right) + \frac{1}{\nu^2(1 - z^2)}e^{-2 \nu \rho}o(1), \tag{3.38} \label{eq:338}
\end{equation*}
as $\nu \to \infty$, uniformly for $z \in U_\nu^\delta$. 

To prove \eqref{eq:338}, we first use \eqref{eq:334} to see that 
\begin{align*}
    U(\zeta, \nu) =&\; \zeta^{1/4}\nu^{1/6}Ai(\nu^{2/3}\zeta)\left(1 + \frac{A_1(\zeta)}{\nu^2}  \right) + \frac{\zeta^{1/4}Ai'(\nu^{2/3}\zeta)}{\nu^{7/6}}B_0(\zeta) + \frac{1}{\nu^2(1-z^2)}e^{-\nu \rho}o(1),\\
    W(\zeta, \nu) =&\; \frac{\zeta^{-1/4}Ai(\nu^{2/3}\zeta)}{\nu^{5/6}}C_0(\zeta) + \frac{\zeta^{-1/4}Ai'(\nu^{2/3}\zeta)}{\nu^{1/6}}\left(1 + \frac{D_1(\zeta)}{\nu^2} \right) + \frac{1}{\nu^2(1 - z^2)}e^{-\nu \rho}o(1),
\end{align*}
where the remainder bounds are obtained using Proposition \ref{prop:3.1}. Using once again the agreement of the leading behavior of \eqref{eq:312} and \eqref{eq:314}, and of \eqref{eq:313} and \eqref{eq:315}, it then follows that 
\begin{equation*}
    U(\zeta, \nu)W(\tilde \zeta, \nu) - U(\tilde \zeta, \nu)W(\zeta, \nu) = \hat{U}(\zeta, \nu)\hat{W}(\tilde \zeta, \nu) - \hat{U}(\tilde \zeta, \nu)\hat{W}(\zeta, \nu) + \frac{1}{\nu^2(1 - z^2)}e^{-2\nu \rho}o(1)
\end{equation*}
as $\nu \to \infty$, uniformly in $z \in U_\nu^\delta$, where
\begin{align*}
    \hat{U}(\zeta, \nu) =&\; \zeta^{1/4}\nu^{1/6}Ai(\nu^{2/3} \zeta) + \zeta^{1/4}\nu^{1/6}Ai(\nu^{2/3}\zeta) \frac{A_1(\zeta)}{\nu^2} + \frac{\zeta^{1/4}Ai'(\nu^{2/3}\zeta)}{\nu^{7/6}}B_0(\zeta)\\
    \hat{W}(\zeta, \nu) =&\; \frac{\zeta^{-1/4}Ai(\nu^{2/3}\zeta)}{\nu^{5/6}}C_0(\zeta) + \frac{\zeta^{-1/4}Ai'(\nu^{2/3}\zeta)}{\nu^{1/6}} + \frac{\zeta^{-1/4}Ai'(\nu^{2/3}\zeta)}{\nu^{1/6}}\cdot \frac{D_1(\zeta)}{\nu^2}.
\end{align*}

Now, applying \eqref{eq:319} and \eqref{eq:320} and noting that $\nu^{-2}(1-z^2)^{-1} = O(\nu^{-4/3})$ in $U_\nu^\delta$, it is straightforward to arrange 
\begin{align*}
    \tilde \zeta^{1/4}\nu^{1/6}Ai(\nu^{2/3}\tilde \zeta)=&\; \zeta^{1/4}\nu^{1/6}Ai(\nu^{2/3} \zeta) - \frac{V_0}{8\nu^{11/6}z} \cdot \frac{\zeta'}{\zeta^{3/4}}Ai(\nu^{2/3}\zeta) \\
    -&\frac{V_0}{2\nu^{7/6}z}\cdot \zeta' \zeta^{1/4}Ai'(\nu^{2/3} \zeta) + \frac{1}{\nu^2(1-z^2)}e^{-\nu\rho}o(1). \tag{3.39} \label{eq:339}
\end{align*}
Whence, using also \eqref{eq:322}
\begin{equation*}
    \tilde\zeta^{1/4} \nu^{1/6}Ai(\nu^{2/3}\tilde \zeta) \frac{A_1(\tilde \zeta)}{\nu^2} = \zeta^{1/4}\nu^{1/6} Ai(\nu^{2/3} \zeta) \frac{A_1(\zeta)}{\nu^2} + \frac{1}{\nu^2(1-z^2)}e^{- \nu \rho}o(1) \tag{3.40} \label{eq:340}
\end{equation*}
By \eqref{eq:319}, \eqref{eq:321}, and \eqref{eq:323}, we have 
\begin{equation*}
    \frac{\tilde \zeta^{1/4} Ai'(\nu^{2/3}\tilde \zeta)}{\nu^{7/6}}B_0(\tilde \zeta) = \frac{\zeta^{1/4}Ai'(\nu^{2/3}\zeta)}{\nu^{7/6}}B_0(\zeta) + \frac{1}{\nu^2(1-z^2)}e^{-\nu \rho}o(1) \tag{3.41} \label{eq:341}
\end{equation*}
Thus, \eqref{eq:339}, \eqref{eq:340}, \eqref{eq:341} show that 
\begin{equation*}
    \hat{U}(\tilde \zeta, \nu) = \hat{U}(\zeta, \nu) - \frac{V_0}{8\nu^{11/6}z}\cdot \frac{\zeta'}{\zeta^{3/4}}Ai(\nu^{2/3} \zeta) - \frac{V_0}{2\nu^{7/6}z}\cdot \zeta' \zeta^{1/4}Ai'(\nu^{2/3} \zeta) + \frac{1}{\nu^2(1-z^2)}e^{-\nu\rho}o(1) \tag{3.42}\label{eq:342}
\end{equation*}
as $\nu \to \infty$, uniformly in $z \in U_\nu^\delta$, In analogous fashion, we also obtain 
\begin{equation*}
    \hat{W}(\tilde \zeta, \nu) = \hat{W}(\zeta, \nu) + \frac{V_0}{8\nu^{13/6}z} \cdot \frac{\zeta'}{\zeta^{5/4}}Ai'(\nu^{2/3}\zeta) - \frac{V_0}{2\nu^{5/6}z}\cdot \zeta' \zeta^{3/4} Ai(\nu^{2/3}\zeta) + \frac{1}{\nu^2(1 - z^2)}e^{-\nu \rho}o(1) \tag{3.43} \label{eq:343}
\end{equation*}
as $\nu \to \infty$, uniformly in $z \in U_\nu^\delta$. 

Using that 
\begin{align*}
    \hat{U}(\zeta, \nu) =&\; \frac{1}{2\pi^{1/2}}e^{-\nu \rho}(1 + o(1))\\
    \hat{W}(\zeta, \nu) =&\; - \frac{1}{2\pi^{1/2}}e^{-\nu \rho}(1 + o(1))
\end{align*}
as $\nu \to \infty$, uniformly in $U_\nu^\delta$, along with \eqref{eq:342}, \eqref{eq:343}, and making some bounds, we see that to prove \eqref{eq:338} it remains to prove that 
\begin{equation*}
    H(\zeta, \nu) = O\left( \frac{1}{\nu} \right) + \frac{1}{\nu^2(1 - z^2)}e^{-2\nu\rho}o(1) \quad \text{as $\nu \to \infty$, uniformly in } z \in U_\nu^\delta \tag{3.44}\label{eq:344}
\end{equation*}
where 
\begin{equation*}
    H(\zeta, \nu) := \frac{V_0}{2\nu^{4/3}z}\zeta' Ai'(\nu^{2/3}\zeta)^2 - \frac{V_0}{2\nu^{2/3} z}\zeta' \zeta Ai(\nu^{2/3}\zeta)^2 + \frac{V_0}{4\nu^2z} \frac{\zeta'}{\zeta}Ai(\nu^{2/3}\zeta)Ai'(\nu^{2/3}\zeta). \tag{3.45} \label{eq:345}
\end{equation*}

We shall show that \eqref{eq:344} holds in $U_1$ and $U_2$ separately, and we begin with $U_1$. In $U_1$, the expansions \eqref{eq:314}, \eqref{eq:315} are relevant, and from \eqref{eq:345} we obtain 
\begin{align*}
    H(\zeta, \nu) =&\; \frac{iV_0\zeta'\zeta^{1/2}}{2\pi \nu z} \left[\sin^2\left( -i \nu \rho - \frac{\pi}{4} \right) - \frac{2b_0'}{i \nu \rho} \cos \left(-i\nu \rho - \frac{\pi}{4} \right)\sin \left(-i\nu\rho - \frac{\pi}{4} \right) \right]\\
    +&\;\frac{iV_0\zeta'\zeta^{1/2}}{2\pi \nu z} \left[\cos^2\left( -i \nu \rho - \frac{\pi}{4} \right) - \frac{2b_0}{i \nu \rho} \cos \left(-i\nu \rho - \frac{\pi}{4} \right)\sin \left(-i\nu\rho - \frac{\pi}{4} \right) \right]\\
    +&\;\frac{V_0\zeta'}{4\pi \nu^2 z \zeta} \cos\left( -i \nu \rho - \frac{\pi}{4} \right)\sin\left(-i\nu\rho - \frac{\pi}{4} \right) + \frac{1}{\nu^2(1-z^2)}e^{-2\nu\rho}o(1)\\
    =&\; O\left( \frac{1}{\nu}\right) + \frac{1}{\nu^2(1-z^2)}e^{-2\nu\rho}o(1).
\end{align*}
In the last step we used that $b_0 +b_0' = \frac{1}{6}$ (see \eqref{eq:316}) and $\rho = \frac{2}{3}\zeta^{3/2}$. 

The proof that \eqref{eq:344} holds in $U_2$ is similar. By using \eqref{eq:312}, \eqref{eq:313} we have here 
\begin{align*}
    H(\zeta, \nu) =&\; \frac{V_0 \zeta' \zeta^{1/2}}{8\pi \nu z}e^{-2\nu \rho} \left(1 + \frac{2u_1'}{\nu \rho} \right) - \frac{V_0\zeta' \zeta^{1/2}}{8\pi \nu z}e^{-2\nu \rho}\left( 1- \frac{2u_1}{\nu \rho} \right)\\
    -&\; \frac{V_0}{16\nu^2 z}\cdot \frac{\zeta'}{\zeta}e^{-2\nu \rho} + \frac{1}{\nu^2(1-z^2)}e^{-2\nu \rho}o(1) \\
    =&\; \frac{1}{\nu^2 (1 - z^2)}e^{-2\nu\rho}o(1)
\end{align*}
where we used $u_1 + u_1' = \frac{1}{6}$ (see \eqref{eq:316}) and $\rho = \frac{2}{3}\zeta^{2/3}$. 

Thus \eqref{eq:344} holds in all of $U_\nu^\delta$, and putting all of the pieces together we conclude the proof of Proposition \ref{prop:3.2}.

\section{PROOFS OF THEOREMS \ref{thm:1.5}, \ref{thm:1.2}, and \ref{thm:1.1}}\label{sec:4}

\subsection{Proof of Theorem \ref{thm:1.5}}\label{sec:4.1} We recall some definitions from the theory of functions of several complex variables; see \cite{11}, \cite{14} for more information. Let $\Omega \subset \C^p$ be an open connected set. A function $f \colon \Omega \to \R \cup \set{-\infty}$ is called \emph{plurisubharmonic} (or p.s.h.) provided $f$ is upper semi-continuous, $f \not\equiv -\infty$ and for any $z \in \Omega$ and any $\omega, r$ such that $z + u \omega \subset \Omega$ for $u \in \C, \abs{u} \leq r$, we have 
\begin{equation*}
    f(z) \leq \frac{1}{2\pi} \int_0^{2\pi} f(z + \omega re^{i \theta}) d\theta.
\end{equation*}
We call a set $E \subset \Omega$ \emph{pluripolar} if there exists a p.s.h. function $f$ such that $E \subset \set{z : f(z) = -\infty}$. A pluripolar set has Lebesgue measure zero and the countable union of pluripolar sets is pluripolar. 

To prove Theorem \ref{thm:1.5}, which is the even-$d$ analogue of \cite[Theorem 1.2]{7}, we will follow \cite[Section 5]{7}. We recall a couple of important lemmas used there and refer to \cite{7} for the simple proofs. 

\begin{lm}[{\cite[Lemma 5.2]{7}}]\label{lem:4.1} Let $\Phi_n, n = 1, 2, \ldots $ be a sequence of p.s.h. functions on an open connected set $\Omega \subset \C^p$. Assume there are constants $c > 0$ and $\gamma > 1$ such that $\Phi_n \leq cn^{-\gamma}$ on $\Omega$ and $\Phi_n(\theta_0) > -cn^{-\gamma}$ for some point $\theta_0 \in \Omega$. Then for every $\alpha < \gamma - 1$ there exists a pluripolar set $E \subset \Omega$ such that $\Phi_n(\theta) = o(n^{-\alpha})$ for every $\theta \in \Omega \backslash E$. 
\end{lm}

The next lemma relates the asymptotics of $n_{V, m}$ and $N_{V, m}$. It is an improvement over \cite[Lemma 1]{19}.

\begin{lm}[{\cite[Proposition 5.1]{7}}]\label{lem:4.2}
    Let $V \in L_c^\infty(\R^d; \C)$, fix $m \in \Z \backslash \set{0}$, and let $\delta, A$ be strictly positive constants such that $\delta < d$. Then we have $(a) \implies (b) \implies (c)$ where 
    \begin{enumerate}[(a)]
        \item $n_{V, m}(r) = Ar^d + O(r^{d - \delta})$ as $r \to \infty$,

        \item $N_{V, m}(r) = \frac{A}{d}r^d + O(r^{d - \delta})$ as $r \to \infty$,

        \item $n_{V, m}(r) = Ar^d + O(r^{d - \frac{\delta}{2}})$ as $r \to \infty$.
    \end{enumerate}
\end{lm}

Now let $\set{V_z}_{z \in \Omega} \subset L^\infty(\B_a; \C)$ be a family analytic in $\Omega$. As in Section \ref{sec:2}, we now define 
\begin{equation*}
    f_{V_z, m}(\lambda) := \det(I + mA_{V_z})
\end{equation*}
and set 
\begin{equation*}
    \psi_m(r, z) := \frac{1}{2\pi r^d} \int_0^\pi \log \abs{f_{V_z, m}(re^{i \theta})}d\theta - \frac{c_d a^d}{d}. \tag{4.1} \label{eq:401}
\end{equation*}
By the argument of \cite[proof of Lemma 3.2]{6}, for large $r$, $\psi_m(r, z)$ is p.s.h. on $\Omega$. 

We may now give the proof of Theorem \ref{thm:1.5}:

\begin{proof}[Proof of Theorem \ref{thm:1.5}]
    Let $\Omega \subset \C^p$ be an open connected set, let $\set{V_z}_{z \in \Omega} \subset L^\infty(\B_a; \C)$ be a family analytic in $\Omega$, and assume there is $m \in \Z \backslash \set{0}$, $z_m \in \Omega$ and $0 < \delta_m \leq 1$ such that \eqref{eq:104} holds. Define the sets $\Omega_j, j = 1, 2, 3, \ldots$, to be the connected component of $\Omega \cap B(z_m, j)$ that contains $z_m$ (where $B(z_m, j)$ is the open ball in $\C^p$ with center $z_m$ and radius $j$). 

    By our remarks above, the function $\psi_m(r, z)$ defined in \eqref{eq:401} is p.s.h. in $\Omega$. Now fix some $j \geq 1$. Using Theorem \ref{thm:2.3} with $\Omega' = \Omega_j$, we have 
    \begin{equation*}
        \psi_m(r, z) \lesssim \frac{\log r}{r}, \quad z \in \Omega_j \tag{4.2} \label{eq:402}
    \end{equation*}
    with implied constant depending upon $m, j$. On the other hand, using \eqref{eq:104}, Lemma \ref{lem:4.2}, and Theorem \ref{thm:2.3} at the fixed $z_m$, we arrange for any $\varepsilon > 0$, 
    \begin{equation*}
        \psi_m(r, z_m) \gtrsim r^{-\delta_m + \varepsilon}. \tag{4.3} \label{eq:403}
    \end{equation*}

    Using these bounds, we now repeat part of the proof of \cite[Theorem 5.3]{7}: Define for $k = \frac{2}{\delta_m}$, 
    \begin{equation*}
        \Phi_n(z) := \psi(n^k, z),\quad  z \in \Omega_j.
    \end{equation*}
    By \eqref{eq:402}, \eqref{eq:403}
    \begin{equation*}
        \Phi_n(z) \lesssim (\log n)n^{-k} \quad \text{and} \quad \Phi_n(z_m) \gtrsim n^{-k(\delta_m - \varepsilon)}.
    \end{equation*}
    Applying Lemma \ref{lem:4.1}, we see that there exists a pluripolar set $E_{\varepsilon, j}$ such that 
    \begin{equation*}
        \Phi_n(z) = o\left(n^{-k(\delta_m -\varepsilon) + 1 + \varepsilon}\right) = o\left(n^{-1 + \frac{2\varepsilon}{\delta_m} + \varepsilon} \right), \quad z \in \Omega_j \backslash E_{\varepsilon, j}.
    \end{equation*}
    Hence, by Theorem \ref{thm:2.3}, we have for $r = n^k$, 
    \begin{equation*}
        N_{V_z, m}(r) = \frac{c_d a^d}{d}r^d + O\left(r^{d - \frac{\delta_m}{2} + \varepsilon + \frac{\varepsilon \delta_m}{2}} \right), \quad z \in \Omega_j \backslash E_{\varepsilon, j}. \tag{4.4} \label{eq:404}
    \end{equation*}
    Since $N_{V_z, m}(r)$ is increasing in $r$, we deduce that if $n^k \leq r \leq (n+1)^k$, 
    \begin{align*}
        \abs{N_{V_z, m}(r) - \frac{c_d a^d}{d}r^d} \lesssim&\; (n+1)^{kd} - n^{kd} + O\left(r^{d-\frac{\delta_m}{2} + \varepsilon +\frac{\varepsilon \delta_m}{2}} \right)\\
        \lesssim&\; n^{kd-1} + O\left(r^{d - \frac{\delta_m}{2} + \varepsilon + \frac{\varepsilon \delta_m}{2}} \right)\\
        \lesssim&\; O \left( r^{d - \frac{\delta_m}{2} + \varepsilon + \frac{\varepsilon\delta_m}{2}} \right).
    \end{align*}
Thus, \eqref{eq:404} holds for $r \to \infty$ with $r \in (0, \infty)$. 

Now define $E_j := \bigcup_{n=1}^\infty E_{1/n, j}$. This set is pluripolar, and for any $\varepsilon > 0$, 
\begin{equation*}
    N_{V_z, m}(r) = \frac{c_d a^d}{d}r^d + O \left(r^{d - \frac{\delta_m}{2} + \varepsilon} \right) \quad \text{as } r \to \infty, z \in \Omega_j \backslash E_j.
\end{equation*}
Setting $\widehat{E}_m := \bigcup_{j=1}^\infty E_j$, which is also pluripolar, we see that for any $\varepsilon > 0$, 
\begin{equation*}
    N_{V_z, m}(r) = \frac{c_d a^d}{d}r^d + O\left( r^{d - \frac{\delta_m}{2} + \varepsilon} \right) \quad \text{as } r \to \infty, z \in \Omega \backslash \widehat{E}_m.
\end{equation*}
Applying Lemma \ref{lem:4.2} concludes the first part of the proof of Theorem \ref{thm:1.5}. 

To prove the second statement, we simply repeat the first part of the proof for each $m \in \Z \backslash \set{0}$ to obtain pluripolar sets $\widehat{E}_m$. Setting $E := \bigcup_{m \in \Z \backslash \set{0}} \widehat{E}_m$, which is pluripolar, the claim holds for this $E$. 
\end{proof}

\subsection{Maximal resonance density for generic potentials}\label{sec:4.2} In this subsection we give the proofs of Theorems \ref{thm:1.1} and \ref{thm:1.2}, and we begin with the more difficult Theorem \ref{thm:1.2}. We shall follow the arguments of Christiansen-Hislop in \cite{5} and \cite[Section 4]{6} but with modifications suited to our purposes. Although some of the proofs are similar, we repeat all arguments in order to illustrate the necessary modifications, for clarity and convenience of the reader. 

Recall that for $V \in L^\infty(\B_a, \C), m \in \Z \backslash \set{0}$, the function $f_{V, m}$ defined in \eqref{eq:202} is analytic in a neighborhood of $\set{\lambda \in \Lambda: 0 \leq \arg \lambda \leq \pi, \abs{\lambda} > c_0 (\norm{V}_\infty)}$. Then for $N, M , \varepsilon > 0, j > 2Nc_0, m \in \Z \backslash\set{0}$, and $\F = \R$ or $\C$, define the sets
\begin{align*}
    A_m(N, M, \varepsilon, j) := \bigg \{V \in L^\infty (\B_a; \F): \norm{V}_\infty \leq N, \int_0^\pi \log \abs{f_{V, m}(re^{i \theta})}d\theta \leq \\
    \left( \frac{2\pi c_da^d}{d} - \varepsilon \right) r^d + Mr^{d-1} \text{ for } 2Nc_0 \leq r \leq j\bigg\}.
\end{align*}

\begin{lm}\label{lem:4.3}
    The sets $A_m(N, M, \varepsilon, j) \subset L^\infty(\B_a; \F)$ are closed. 
\end{lm}
\begin{proof}
    Let $V_k \in A_m(N, M, \varepsilon, j)$ with $V_k \to V$ in $L^\infty$ norm. Then clearly $\norm{V}_\infty \leq N$. The proof of \cite[Lemma 4.1]{6} shows that $f_{V_k, m}(\lambda) \to f_{V, m}(\lambda)$ as $k \to \infty$ uniformly in $0 \leq \arg \lambda \leq \pi, 2Nc_0 \leq \abs{\lambda} \leq j$. It follows that 
    \begin{equation*}
        \int_0^\pi \log \abs{f_{V, m}(re^{i \theta})} d\theta = \lim_{k \to \infty} \int_0^\pi \log \abs{f_{V_k, m}(re^{i \theta})} d\theta \leq \left( \frac{2 \pi c_d a^d}{d} - \varepsilon \right)r^d + Mr^{d-1}
    \end{equation*}
    for all $2Nc_0 \leq r \leq j$, and thus $V \in A_m(N, M, \varepsilon, j)$ so the claim follows. 
\end{proof}

For $N, M, \varepsilon > 0, m \in \Z \backslash \set{0}$, we now define 
\begin{equation*}
    B_m(N, M, \varepsilon) := \bigcap_{j \geq2Nc_0}A_m(N, M, \varepsilon, j),
\end{equation*}
which is closed by Lemma \ref{lem:4.3}. Our next lemma gives a characterization of those $V \in L^\infty(\B_a; \F)$ for which the maximal resonance density is not achieved. 

\begin{lm}\label{lem:4.4}
    If $V \in L^\infty(\B_a; \F)$ for $\F = \R$ or $\C$, and 
    \begin{equation*}
        \varlimsup_{r \to \infty} \frac{N_{V, m}(r)}{r^d} < \frac{c_d a^d}{d},
    \end{equation*}
    then there exists $N, M, \ell \in \N$ such that $V \in B_m(N, M, \frac{1}{\ell})$. 
\end{lm}
\begin{proof}
    Under the hypothesis of the lemma, there exists $\varepsilon > 0$ such that 
    \begin{equation*}
        N_{V, m}(r) \leq \left( \frac{c_d a^d}{d} - \varepsilon\right) r^d
    \end{equation*}
    for large enough $r$. By Theorem \ref{thm:2.3}, there exists a constant $A_m > 0$ with 
    \begin{equation*}
        \int_0^\pi \log \abs{f_{V, m}(re^{i \theta})} d\theta \leq \left( \frac{2\pi c_d a^d}{d} - 2\pi \varepsilon \right)r^d + 2\pi A_mr^{d-1}
    \end{equation*}
    for large $r$. Thus, by taking a large enough integer $M > 0$, we have 
    \begin{equation*}
        \int_0^\pi \log \abs{f_{V, m}(re^{i \theta})}d \theta \leq \left( \frac{2\pi c_da^d}{d} - 2\pi \varepsilon \right)r^d + Mr^{d-1}
    \end{equation*}
    for all $r \geq 2Nc_0$, where $N$ is any integer with $\norm{V}_\infty \leq N$. It follows that for integer $\ell > \frac{1}{2\pi \varepsilon}, V \in B_m(N, M, 1/\ell)$ and the lemma follows. 
\end{proof}

Recall the definition of the set $\mathfrak M_{\F, a}$ from Theorem \ref{thm:1.2} in the introduction. The following gives one half of the proof of that theorem. 

\begin{lm}\label{lem:4.5}
For any $a > 0, m \in \Z \backslash \set{0}$, and $\F = \R$ or $\C$, the set 
\begin{equation*}
    \mathfrak M_{\F, a, m} := \set{V \in L^\infty(\B_a; \F) : \varlimsup_{r \to \infty} \frac{N_{V, m}(r)}{r^d} = \frac{c_d a^d}{d}}
\end{equation*}
is a $G_\delta$ set, and thus 
\begin{equation*}
    \mathfrak M_{\F, a} = \bigcap_{m \in \Z \backslash \set{0}} \mathfrak M_{\F, a, m}
\end{equation*}
is as well. 
\end{lm}
\begin{proof}
    By Lemma \ref{lem:4.4}, the complement of $\mathfrak M_{\F, a, m}$ is contained in 
    \begin{equation*}
        \bigcup_{(N, M, \ell) \in \N^3} B_m\left(N, M, \frac{1}{\ell} \right)
    \end{equation*}
    which is an $F_\sigma$ set, i.e., a countable union of closed sets. On the other hand, if $V \in B_m(N, M, 1/\ell)$ for some $N, M, \ell \in \N$, then by Theorem \ref{thm:2.3}, there is $A_m > 0$ so that for large enough $r$, 
    \begin{align*}
        N_{V, m}(r) \leq &\; \frac{1}{2\pi} \int_0^\pi \log \abs{f_{V, m}(re^{i \theta})} d\theta + A_mr^{d-1}\\
        \leq &\; \left( \frac{c_da^d}{d} - \frac{1}{2\pi \ell} \right)r^d + (M + A_m)r^{d-1}
    \end{align*}
    and thus $V \notin \mathfrak M_{\F, a, m}$. Thus $\mathfrak M_{\F, a, m}$ is the complement of an $F_\sigma$ set. 
\end{proof}

We may now give the proof of Theorem \ref{thm:1.2}:

\begin{proof}[Proof of Theorem \ref{thm:1.2}]
    By Lemma \ref{lem:4.5} it remains to show that $\mathfrak M_{\F, a}$ is dense in $L^\infty(\B_a; \F)$. We follow the ideas of \cite[Corollary 1.3]{1}, as in \cite[proof of Theorem 1.1]{6}, but with modifications suitable to our purposes. 

    Let $V \in L^\infty(\B_a; \F)$ be given. Let $V_1(x) := V_0 \chi_{[0, a]}(\abs{x}), V_0 > 0$, be as in Theorem \ref{thm:1.4}, and define the family of potentials 
    \begin{equation*}
        V_z(x) := zV_1(x) + (1-z)V(x), \quad z \in \C.
    \end{equation*}
    Applying Theorem \ref{thm:1.5} to $\set{V_z}_{z \in \C}$, we find a pluripolar set $E \subset \C$ such that for every $m \in \Z \backslash \set{0}$ and $z \in \C \backslash E$,
    \begin{equation*}
        n_{V_z, m}(r) = c_da^dr^d + O\left(r^{d - \frac{3}{16} + \varepsilon} \right) \quad \text{as } r \to \infty, \text{ for any } \varepsilon >0.
    \end{equation*}
By Lemma \ref{lem:4.2} this implies that for every $m \in \Z \backslash\set{0}$ and $z \in \C \backslash E$,
\begin{equation*}
    N_{V_z, m}(r) = \frac{c_d a^d}{d}r^d + O\left(r^{d - \frac{3}{16} + \varepsilon}\right) \quad \text{as $r \to \infty$, for any } \varepsilon > 0
\end{equation*}
    
    In particular, $\set{V_z}_{z \in \C \backslash E} \subset \mathfrak M_{\F, a}$. But recalling that $E \cap \R$ has $\R$-Lebesgue measure zero (e.g. \cite[Section 12.2]{14}), given $\delta > 0$ we may find $z_0 \in \R, z_0 \notin E$ with $\abs{z_0} < \delta/(1 + \norm{V}_\infty + \norm{V_1}_\infty)$. Then $V_{z_0} \in \mathfrak M_{\F, a}$ and $\norm{V_{z_0} - V}_\infty < \delta$. If $V$ is real-valued, so if $V_{z_0}$. Thus $\mathfrak M_{\F, a}$ is dense in $L^\infty(\B_a; \F)$ as claimed. 
\end{proof}

The modifications required for the proof of the odd-$d$ version, Theorem \ref{thm:1.1}, are as follows: Let $s_V(\lambda) := \det S_V(\lambda)$ be the determinant of the scattering matrix; when the dimension is odd (as we are now assuming), $s_V(\lambda)$ is meromorphic in $\C$, and with finitely many exceptions $s_V(\lambda_0) = 0$ if and only if $-\lambda_0$ is a resonance. In place of $A_m(N, M, \varepsilon, j)$ we define 
\begin{align*}
    A(N, M, \varepsilon, j) := &\; \bigg\{ V \in L^\infty(\B_a; \F) : \norm{V}_\infty \leq N, \int_0^\pi \log \abs{s_V(re^{i \theta})} d\theta \leq \\
    &\; \left(\frac{2\pi c_d a^d}{d} - \varepsilon \right)r^d + Mr^{d-1} \text{ for } 2N+1 \leq r \leq j \bigg\}.
\end{align*}
The proof of \cite[Lemma 2.1]{5} may be used to see that $V_k \to V$ implies $s_{V_k}(\lambda) \to s_V(\lambda)$ uniformly in $\operatorname{Im} \lambda \geq 0, 2N+1 \leq \abs{\lambda} \leq j$, and following our proof of Lemma \ref{lem:4.3} shows that $A(N, M, \varepsilon, j)$ is closed. In place of $B_m(N, M, \varepsilon)$, we then set 
\begin{equation*}
    B(N, M, \varepsilon) := \bigcap_{j \geq 2N + 1}A(N, M, \varepsilon, j)
\end{equation*}
which is also closed. 

Analogues of Lemmas \ref{lem:4.4} and \ref{lem:4.5} are had with essentially no change to the proofs, using \cite[Proposition 3.2]{7} in place of Theorem \ref{thm:2.3}, Theorem \ref{thm:1.1} follows by straightforward modification of the proof of Theorem \ref{thm:1.2}, using \cite[Theorems 1.1 and 1.2]{7} in place of our Theorems \ref{thm:1.4} and \ref{thm:1.5}.

\bibliographystyle{amsplain}

\bibliography{references}

\noindent\emph{Email address}: travisdcunningham@gmail.com

\end{document}